\documentclass{article}
\usepackage{graphicx} 
\usepackage{tikz}
\usetikzlibrary{positioning}
\usepackage{float}
\usepackage{amssymb}
\usepackage{soul}
\usepackage{url}

\usepackage{amsthm}
\usepackage{amsfonts}
\usepackage{amsmath}
\usepackage [english]{babel}
\usepackage [autostyle, english = american]{csquotes}
\MakeOuterQuote{"}
\usepackage[a4paper, margin=1in]{geometry}

\theoremstyle{definition}
\newtheorem{theorem}{Theorem}
\newtheorem{lemma}{Lemma}
\newtheorem{sub-lemma}{Sub-lemma}
\newtheorem{proposition}{Proposition}
\newtheorem{definition}{Definition}
\newtheorem{remark}{Remark}

\usepackage{sectsty}
\usepackage{titlesec}

\titleformat{\section}
  {\normalfont\centering}
  {\thesection}
  {1em}
  {\MakeUppercase}

\begin{document}
\title{\normalsize \textbf{\MakeUppercase{Tangent Space of the Stable and Unstable Manifold of Anosov Diffeomorphism on 2-Torus}}}

\author{Federico Bonneto, Jack Wang, Vishal Kumar}
\date{June 2024}

\begin{abstract}
    In this paper we describe the tangent vectors of the stable and unstable manifold of a class of Anosov diffeomorphisms on the torus $\mathbb{T}^2$ using the method of formal series and derivative trees. We start with linear automorphism that is hyperbolic and whose eigenvectors are orthogonal. Then we study the perturbation of such maps by trigonometric polynomial.  It is known \cite{LinearOperators} that there exist a (continuous) map $H$ which acts as a change of coordinate between the perturbed and unperturbed system, but such a map is in general, not differentiable. By "re-scaling" the parametrization $H$, we will be able to obtain the explicit formula for the tangent vectors of these maps.
    \end{abstract}
    
\maketitle

\section{Introduction}
Let $\mathbb{T}^2 = \mathbb{R}^2/2 \pi \mathbb{Z}^2$ be the two dimensional torus. The map $S: \mathbb{T}^2 \rightarrow \mathbb{T}^2$ defined as
    $$S (\psi) = S_0 \psi \pmod{2\pi}, \text{ where } S_0 = \begin{bmatrix}
1 &1 \\ 
1 &0 
\end{bmatrix},$$
has eigenvalues 
$$\lambda_\pm = \frac{1 \pm \sqrt{5}}{2}.$$ Let $v_{\pm}$ be the (normalized) eigenvectors of $S_0$ associated to each eigenvectors \footnote{Note that the theory developed in this paper remains valid for any diagonalizable integer valued matrix with determinant 1 such that the two eigenvalues has modules strictly greater and less than 1, respectively, but we will perform the computations on  this particular example.}. Note that $v_+$ and $v_-$ are orthogonal. It is easy to verify that $S$ is a diffeomorphism. The map $S$ is a simple example of an Anosov diffeomorphism:

\begin{definition}
\label{Anosov}
    Let $M$ be a manifold together with a metric $|| \cdot ||$. A diffeomorphism $F:M \rightarrow M$ is said to be $\textbf{Anosov}$ if \\
        \indent(a) There exist two continuous vector bundles $E^u$ and $E^s$, called the unstable and stable bundle, such that for each $p \in M$, one has $T_pM = E^u_p \oplus E^s_p$; \\
        \indent(b) The bundles $E^u$ and $E^s$ are invariant under the differential $dF$, in the sense that $(dF)_p E^u(p) = E^u(F(p))$ and $(dF)_p E^s(p) = E^s(F(p))$ for all $p \in M;$\\
        \indent(c) (Hyperbolicity) There exist $C>0$ and $\lambda \in (0,1)$ such that for all $n \in \mathbb{N}, p \in M,$ we have 
        \[||(dF^n)_p v|| \leq C \lambda^n ||v|| \text{  if  } v \in E^s(p);\]
        \[||(dF^{-n})_p v|| \leq C \lambda^n ||v|| \text{  if  } v \in E^u(p).\]
        \indent(d) (Density) there exists some point $p \in M$ such that the set $\{ F^{n}(p) | n \in \mathbb{Z}\}$ is dense in $M$.
\end{definition}
Because $|\lambda_+| > 1$ and $|\lambda_-|<1$,  $S$ is Anosov simply by announcing the splitting of tangent spaces to be in the two directions of the corresponding eigenvectors $v_{\pm}$. Once the splitting is determined, item (b) and (c) in definition \ref{Anosov} are trivially satisfied. A crucial fact about Anosov diffeomorphism is the following well-known property, whose proof could be found in \cite{wen}:

\begin{theorem}
    With the notation as in Definition \ref{Anosov}, the stable and unstable set (also called stable and unstable manifold) on $M$ under $F$ relative to $x\in M$, defined respectively as
    \[W^s(x) := \{y \in M: d(F^ny, F^nx) \rightarrow 0 \text{ as } n \rightarrow \infty\}\]
    \[W^u(x):= \{y \in M: d(F^{-n}y, F^{-n}x) \rightarrow 0 \text{ as } n \rightarrow \infty\}\]
    are manifolds with regularity as high as $F$. Furthermore, for every $p\in M$, one has
    \[T_p W^s(p) = E^s(p), T_p W^u(p) = E^u(p).\]
\end{theorem}

This theorem shows that, if one choose a (continuous) vector field $v^s(p)$ on $M$ so that $v^s(p) \in E^s(p)$ for every $p \in M$, then the integral curve generated by $v^s(p)$ is simply the stable manifold, provided that we have uniqueness of the flow. The same is true for the unstable manifold. For example, it is easy to check that the stable and unstable manifolds of $S$ are given by $W^s_0(\psi) = \{\psi+tv_- \pmod{2\pi} : t \in \mathbb{R}\}$ and $W^u_0(\psi) = \{\psi+tv_+ \pmod{2\pi} : t \in \mathbb{R}\}$, respectively. \\

Now consider a small perturbation $S_{\varepsilon}$ of $S$ on $\mathbb{T}^2$ given by
\[S_{\varepsilon} (v)=S_0 (v)-\varepsilon f(v) \pmod{2\pi}\]
for some small $\varepsilon>0$, where $f$ is a trigonometric polynomial $f(v)=\sum_{\textbf{n} \in \mathbb{Z}^2, |\textbf{n}|<N} \textbf{c}_ne^{i\textbf{n} \cdot v}$. Since $f$ is $2\pi$ periodic, $S_{\varepsilon}$ is well defined. It is known that a $C^1$ perturbation of an Anosov Diffeomorphism remains Anosov. (See \cite{cooper2021} for proof). We will derive an explicit formula for the tangent vectors of the stable and unstable manifold of $S_{\epsilon}.$ \\
\indent For $\varepsilon$ small enough, it is possible to construct a continuous "change of coordinate" between $S$ and $S_{\varepsilon}$, known as structural stability 
 \cite{LinearOperators}: 

\begin{proposition}
\label{conjugation theorem}
With the notation as above, for every $\beta \in (0,1),$ there exist constants $C < \infty$ and $\varepsilon_0 > 0$ (both depends on $\beta$) such that for each $\varepsilon \in (-\varepsilon_0, +\varepsilon_0)$ there exist a unique homeomorphism $H_{\varepsilon}:\mathbb{T}^2 \rightarrow \mathbb{T}^2$ satisfying 
\[H_{\varepsilon} S H_{\varepsilon}^{-1} = S_{\varepsilon}, H_0=\text{id},\]
with the property that\\ 
\indent (i) $H_{\varepsilon} (\varphi)$ is Hölder continuous in $\varphi$ with exponent $\beta$ and modulus $C$;\\
\indent (ii) the correspondence $\varepsilon \rightarrow H_{\varepsilon}(\varphi)$ is analytic for all $\varphi \in \mathbb{T}^2$. \\
\indent (iii) Denote $W^s_{\epsilon}(x),W^u_{\epsilon}(x)$ the stable and unstable manifold of $S_{\epsilon}$. Then $W^s_{\varepsilon}(H_{\varepsilon}(\psi))=H_{\varepsilon}(W_0^s(\psi));$ $W^u_{\varepsilon}(H_{\varepsilon}(\psi))=H_{\varepsilon}(W_0^u(\psi)).$
\end{proposition} 
    
To give some intuition why one expects (iii) to be true, note that if we pick $y \in H_{\varepsilon}(W_0^s(\psi))$, then $y=H_{\varepsilon}(\psi+t^{*}v)$ for some $t^*$, and 
\[d(S_{\varepsilon}^n y, S_{\varepsilon}^nH_{\varepsilon}(\psi))=|H_{\varepsilon}S^n(\psi+t^* v_{-})-H_{\varepsilon}S^n(\psi)|=|H_{\varepsilon}(S^n\psi+\lambda_{-}^nt^* v_{-})-H_{\varepsilon}S^n(\psi)|,\]
which $\rightarrow 0$ as $n \rightarrow \infty$ (The corresponding property is also true for the unstable set). Furthermore, since $W^s_{0}(\psi)$ (resp.$W^u_{0}(\psi)$) is dense in $\mathbb{T}^2$ (as $v_{\pm}$ has irrational coordinates) and $H_{\varepsilon}$ is continuous, $W^s_{\varepsilon}(\psi)$ (resp. $W^u_{\varepsilon}(\psi)$) is dense in $\mathbb{T}^2$ as well. \\
\indent In what follows, we will sometimes abbreviate $H_{\varepsilon}$ simply as $H$, depending on how much emphasis we would like to put on the number $\varepsilon.$

\section{Discussion of the Main Problem, and the Main Result}
Our goal is to find the tangent vector of the sets $W^s_{\varepsilon}(x)$ and $W^u_{\varepsilon}(x).$ With the help of proposition \ref{conjugation theorem} (iii), the most natural thing to do is to differentiate the parametrization $t \mapsto H_{\varepsilon}(\psi+tv^{\pm})$. However, this is generality not possible. To see why, note that in general we should expect
\begin{equation}
\label{long limit formula}
    \partial_- H(\psi) \text{=}\lim_{t \rightarrow 0} \frac{H(\psi+tv_-)-H(\psi)}{t}= \begin{pmatrix}
        0 \\ 1
    \end{pmatrix}+\lim_{t \rightarrow 0} \sum_{k=1}^{\infty}  \varepsilon^k\begin{pmatrix}
 \frac{\left[h^{(k)}_+(\psi+tv_-)- h^{(k)}_+(\psi)\right ]}{t}\\ 
\frac{\left[h^{(k)}_-(\psi+tv_-)-h^{(k)}_-(\psi) \right ]}{t}
\end{pmatrix},
\end{equation}
where the subscripts $\pm$ of any function represents the component of that function in the direction $v_{\pm}$; $\partial_\pm$ is directional derivative operator in the direction of $v_\pm,$  and $h^{(k)}$  is the $k-$th order term in the expansion $H_{\varepsilon} (\psi)=\psi+\varepsilon h^{(1)}(\psi)+\varepsilon^2 h^{(2)}(\psi)+\varepsilon^3 h^{(3)}(\psi)+...$ Unfortunately, There is no naive way to compute the limit (\ref{long limit formula}). For example, the formula for $h_-^{(1)}$ is
\[h_-^{(1)}(\psi)=\sum_{p=-1}^{-\infty}-\lambda_-^{|p+1|} f_-(S^p\psi).\]
Differentiating in the $v_-$ direction, assuming exchanging limit and summation is possible, we get
$$\partial_{-} h_{-}^{(1)}(\psi)\text{=}\sum_{p=-1}^{-\infty} -\lambda_{-}^{|p + 1|} \frac{d}{dt} f_-(S^{p} (\psi + tv_-)) \bigg|_{t=0}=\sum_{p=-1}^{-\infty} -\lambda_{-}^{|p + 1|} (\partial_- f_-)(S^{p}\psi) \lambda_{-}^p,$$ 
and after combining the two $\lambda_{-}$ factors, we lose t
he converging factor and we are left with a sum that may not converge (unless special cancellation occurs)! Here is an alternative solution:
\begin{theorem}
\label{main theorem}
For $\psi \in \mathbb{T}^2, t>0$, define $$V_{\epsilon}(\psi, t) := \frac{H(\psi + tv_-) - H(\psi)}{t}.$$
Let $V_\alpha(\psi, t) := V_{\epsilon}(\psi,t) \cdot v_\alpha$.
Then there exist $D>0$ such that for $|\epsilon|<D$, the limit 
\begin{equation}
\label{main equation}
   v_{\epsilon}(\psi) := \lim_{t \to 0} \frac{V_{+}(\psi, t)}{V_{-}(\psi, t)} 
\end{equation}
exists.
\end{theorem}

\begin{remark}
\label{proof steps}
    Theorem \ref{main theorem} shows that, as $t \rightarrow 0,$ there exist a "limiting direction" of the quotient $V_{\varepsilon},$ i.e., the vector $(v_{\epsilon}(\psi),1),$ which gives us a good candidate for the tangent space of $W^s_{\varepsilon}.$ The outline of the proof is now as follows. We shall establish that\\
\indent (i) There exist $t^*>0$ and $D>0$ such that there is a unique functions $q(\varepsilon,t)$ satisfying
\begin{equation}
\label{formal cauchy product}
    V_+(\psi,t)=V_-(\psi,t)q(\varepsilon,t)
\end{equation}
such that if $t\in (0, t^*),$ then $q(\varepsilon,t)$ is analytic in $\varepsilon$ for $|\varepsilon|<D;$\\
\indent (ii) Let $q^{(k)}(t)$ be the $k-$th order term of $q(\varepsilon,t)$. Then $\text{Val}[q^{(k)}(0)]:=\lim_{t \rightarrow 0} q^{(k)}(t)$ exist; \\
\indent (iii) Show that $v_{\varepsilon}(\psi):=\sum_{k=0}^{\infty} \varepsilon^k \text{Val}[ q^{(k)}(0)]$ converges absolutely for $|\varepsilon|<D$, and gives the limit to (\ref{main equation}).
\end{remark}
First we find the representation of $q(\varepsilon,t)$ which we will be working with. Treat (\ref{formal cauchy product}) first as a formal Cauchy product of power series in $\varepsilon$. Let $V^{(k)}_{\alpha}(\psi,t)$ be the $k$-th order term of $V_{\alpha}(\psi,t)$. Then (\ref{formal cauchy product}) reads
$$\left(\sum_{k = 1}^{\infty} \epsilon^k V_{+}^{(k)} \right)=\left (\sum_{k = 1}^{\infty} \epsilon^k q_k \right) \left(\sum_{k = 0}^{\infty} \epsilon^k V_{-}^{(k)} \right)  $$
As a formal power series, we have, by lemma \ref{explicit formula for power series} from the Appendix, 
\begin{equation}
\label{q_n(t)}
    q_n(t) = V_{+}^{(n)} + \sum_{k = 1}^{n - 1}V_{+}^{(k)} \left(\sum_{s = 0}^{n - k}(-1)^s \sum_{\substack{m_1 + ... + m_s = n - k \\ m_1,...,m_s \geq 0}} V_-^{(m_1)} V_-^{(m_2)} ... V_-^{(m_s)} \right)
\end{equation}
\section{Motivation For the Stragegy}
We now will begin by addressing item (ii) of Remark \ref{proof steps}. To motivate how this will be done, let us consider the equation for $q_2$:
$$q_2(t) = V_{+}^{(2)} - V_{+}^{(1)}V_{-}^{(1)} = \frac{h_+^{(2)}(\psi+tv_-)-h_+^{(2)}(\psi)}{t}- \frac{h_+^{(1)}(\psi+tv_-)-h_+^{(1)}(\psi)}{t}\frac{h_-^{(1)}(\psi+tv_-)-h_-^{(1)}(\psi)}{t}$$
If we take the limit as $t$ goes to $0$, the three quotient above looks a lot like the "directional directives" of the $h$'s (although they do not converge in general, as mentioned earlier). Thus, we seek to define objects $\partial_- h_{+}^{(2)}, \partial_- h_{+}^{(1)}$ and $\partial_- h_{-}^{(1)}$, and operations $+,-,\cdot $ between them such that
$$\lim_{t \to 0} q_2(t) = \partial_- h_{+}^{(2)} - (\partial_- h_{+}^{(1)})\cdot ( \partial_- h_{-}^{(1)}).$$
The obvious candidate is to define this "derivative" as the formal series which involves taking the formal derivative.\\
\indent For example, if we let $\mathbb{Z}_{+}$ be the set of non-negative integers, and $\mathbb{Z}_{-}$ be the set of negative integers, by results from the next section,
$$h_+^{(2)}(\psi)=\sum_{a\in \mathbb{Z}_+}\sum_{b\in \mathbb{Z}_+} \lambda_+^{-|a+1|} \lambda_+^{-|b+1|} \partial_+ f_+(S_0^{a} \psi) f_+(S_0^{a+b} \psi)+\sum_{a\in \mathbb{Z}_+}\sum_{b\in \mathbb{Z}_-} -\lambda_+^{-|a+1|} \lambda_-^{|b+1|} \partial_- f_+(S_0^{a} \psi) f_-(S_0^{a+b} \psi),$$
we define $\partial_-h^{(2)}_+$ to be the following sum of (formal) double indexed series
\begin{align}
\label{del h^2 +}
    \partial_-h^{(2)}_+ := &\sum_{a\in \mathbb{Z}_+}\sum_{b\in \mathbb{Z}_+} \lambda_+^{-|a+1|} \lambda_+^{-|b+1|} \lambda_-^{a} (\partial_-\partial_+) f_+(S_0^{a} \psi) f_+(S_0^{a+b} \psi) + \notag \\
    & \sum_{a\in \mathbb{Z}_+}\sum_{b\in \mathbb{Z}_+} \lambda_+^{-|a+1|} \lambda_+^{-|b+1|} \lambda_-^{a+b} \partial_+ f_+(S_0^{a} \psi) \partial_-f_+(S_0^{a+b} \psi) + \notag \\
    & \sum_{a\in \mathbb{Z}_+}\sum_{b\in \mathbb{Z}_-} -\lambda_+^{-|a+1|} \lambda_-^{|b+1|} \lambda_-^a \partial_-^2 f_+(S_0^{a} \psi) f_-(S_0^{a+b} \psi) + \notag \\
    & \sum_{a\in \mathbb{Z}_+}\sum_{b\in \mathbb{Z}_-} -\lambda_+^{-|a+1|} \lambda_-^{|b+1|} \lambda_-^{a+b} \partial_- f_+(S_0^{a} \psi) \partial_-f_-(S_0^{a+b} \psi).  
\end{align}
That is, we differentiate within the sum and separate the two sums created by the product rule. Furthermore, the product $(\partial_- h_{+}^{(1)})\cdot ( \partial_- h_{-}^{(1)})$ is defined as the formal series
\begin{align}
\label{del h+ * del h-}
    (\partial_- h_{+}^{(1)})\cdot ( \partial_- h_{-}^{(1)}) & = \left( \sum_{a\in \mathbb{Z}_+}\lambda_+^{-|a+1|} \lambda_-^{a}\partial_-f(S^a_0 \psi)\right) \cdot \left( \sum_{b\in \mathbb{Z}_-}-\lambda_-^{|b+1|} \lambda_-^{b}\partial_-f(S^b_0 \psi)\right) \notag\\
     & := \sum_{a\in \mathbb{Z}_+}\sum_{b\in \mathbb{Z}_-} -\lambda_+^{-|a+1|} \lambda_-^{|b+1|} \lambda_-^{a+b} \partial_- f_+(S_0^{a} \psi) \partial_-f_-(S_0^{b} \psi) \notag\\
     & = \sum_{a\in \mathbb{Z}_+}\sum_{b\in \mathbb{Z}_-} -\lambda_+^{-|a+1|} \lambda_-^{a-1} \partial_- f_+(S_0^{a} \psi) \partial_-f_-(S_0^{b} \psi).
\end{align}
Note that all but the last formal series in (\ref{del h^2 +}) converges absolutely in the ordinary sense. This last series could be break into two parts:
\begin{align}
\label{break into two parts}
    &\sum_{a\in \mathbb{Z}_+}\sum_{b=0}^{-a} -\lambda_+^{-|a+1|} \lambda_-^{|b+1|} \lambda_-^{a+b} \partial_- f_+(S_0^{a} \psi) \partial_-f_-(S_0^{a+b} \psi)+ \notag\\
    &\sum_{a\in \mathbb{Z}_+}\sum_{b=-a-1}^{\infty} -\lambda_+^{-|a+1|} \lambda_-^{|b+1|} \lambda_-^{a+b} \partial_- f_+(S_0^{a} \psi) \partial_-f_-(S_0^{a+b} \psi),
\end{align}
where the first part converges absolutely in the ordinary sense, and the second part, after combing the exponents of the $\lambda$'s and re-indexing,
\begin{align}
   &\sum_{a\in \mathbb{Z}_+}\sum_{b=-a-1}^{\infty} -\lambda_+^{-|a+1|} \lambda_-^{|b+1|} \lambda_-^{a+b} \partial_- f_+(S_0^{a} \psi) \partial_-f_-(S_0^{a+b} \psi) \notag\\
   =& \sum_{a\in \mathbb{Z}_+}\sum_{b=-a-1}^{\infty} -\lambda_+^{-|a+1|}\lambda_-^{a-1} \partial_- f_+(S_0^{a} \psi) \partial_-f_-(S_0^{a+b} \psi) \notag \\
   =& \sum_{a\in \mathbb{Z}_+}\sum_{b\in \mathbb{Z}_-} -\lambda_+^{-|a+1|} \lambda_-^{a-1} \partial_- f_+(S_0^{a} \psi) \partial_-f_-(S_0^{b} \psi).
\end{align}
This is exactly the same series as $(\ref{del h+ * del h-})$! Thus we expect them to "cancel" and will obtain a good candidate for the value of $\partial_- h_{+}^{(2)} - (\partial_- h_{+}^{(1)})\cdot ( \partial_- h_{-}^{(1)}):$ just let it equal to the sum the first three series in (\ref{del h^2 +}) subtracted by the first series in (\ref{break into two parts}) (the actual value that these series converges to). If we could do this at every order, then the last step will be to show that this value agree with the limit we are interested. 

\section{Formula For the conjugation}
\indent Now we present the explicit formula for $H_{\varepsilon}$. For details about the construction, see \cite{LinearOperators}. We first need some terminologies.
\begin{definition}
    Let $n>0$ be a positive integer. A $\textbf{tree}$ is an ordered $k$-tuple $(\vartheta_1,...,\vartheta_k)$ of trees, where $k \geq 0$. We adapt the convention that the only ordered 0-tuple of tree is the pair of parenthesis "$()$" without any content. A $\textbf{node}$ of a tree $\vartheta=(\vartheta_1,...,\vartheta_k)$ is either the tree itself, or a node of one of its elements $\vartheta_i (1 \leq i \leq k)$ if $k \geq 1$. We will call the node represented by the tuple itself of a tree its $\textbf{top node}$. The set of all nodes of a tree $\vartheta$ is denoted as $V(\vartheta)$. The $\textbf{parent}$ of a node $v \in V(\vartheta)$ is the unique node $w$ such that $v$ is an element of the tuple $w$. The top node does not have a parent. If $a,b \in V(\vartheta)$, the partial order $a \succeq b$ represents either $a=b$ or there exist a finite sequence $a_0, a_1,...,a_s \in V(\vartheta)$ satisfying $a_0=a, a_s=b$ with the property that $a_i$ is the parent of $a_{i+1}$ for all $0\leq i \leq s-1$. For $v\in V(v),$ let $s_v$ be the number of immediate child of $v$, i.e., the size of the tuple $v$. Finally, A $\textbf{linear tree}$ is a tree such that, for every node $v$, $s_v \leq 1$.
\end{definition}

Note that we deliberately avoided defining trees in terms of an undirected graph in which any two vertices are connected by exactly one path, in order to emphasize the fact that $((), (()))$ and $((()),())$, for example, are considered two different trees. Of course, there is an obvious correspondence between our definition of tree and the  "vertices and edges" definition, and in practice we will (morally) use this more intuitive image to think of a tree. \\
\indent For example, there are 5 different trees with 4 vertices:
\[(((()))) \hspace{1cm} ((), (())) \hspace{1cm} ((()), ()) \hspace{1cm} ((),(),()) \hspace{1cm} (((), ()))\]
which intuitively corresponds to the following graphical representation:

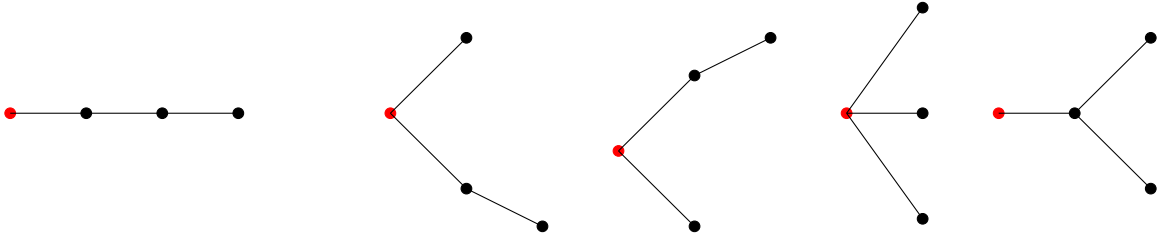
\begin{figure}[h]
    \centering
    \begin{tikzpicture}
    \filldraw[red] (0,0) circle (2pt);
    \filldraw (1,0) circle (2pt);
    \filldraw (2,0) circle (2pt);
    \filldraw (3,0) circle (2pt);
    \draw (0,0) -- (3,0);

    \filldraw[red] (5,0) circle (2pt);
    \filldraw (6,1) circle (2pt);
    \filldraw (6,-1) circle (2pt);
    \filldraw (7,-1.5) circle (2pt);
    \draw (5,0) -- (6,1);
    \draw (5,0) -- (6,-1) -- (7,-1.5);

    \filldraw[red] (8,-0.5) circle (2pt);
    \filldraw (9,0.5) circle (2pt);
    \filldraw (10,1) circle (2pt);
    \filldraw (9,-1.5) circle (2pt);
    \draw (8,-0.5) -- (9,0.5) -- (10,1);
    \draw (8,-0.5) -- (9,-1.5);

    \filldraw[red] (11,0) circle (2pt);
    \filldraw (12,0) circle (2pt);
    \filldraw (12,-1.4) circle (2pt);
    \filldraw (12,1.4) circle (2pt);
    \draw (11,0) -- (12,1.4);
    \draw (11,0) -- (12,0);
    \draw (11,0) -- (12,-1.4);

    \filldraw[red] (13,0) circle (2pt);
    \filldraw (14,0) circle (2pt);
    \filldraw (15,1) circle (2pt);
    \filldraw (15,-1) circle (2pt);
    \draw (13,0) -- (14,0);
    \draw (14,0) -- (15,1);
    \draw (14, 0) -- (15, -1);
\end{tikzpicture}
    \caption{Graphical representation of trees with 4 vertices, with top node colored red.}
    \label{fig:trees}
\end{figure}

\begin{definition} A $\textbf{labeled tree}$ $\vartheta_l$ is a tree such that on each node $v \in V(\vartheta_l)$, there is a label $\alpha_v \in \{+,-\}$, a label $p_v \in \mathbb{Z}_{a_v}$. Two labeled tree are equal iff they are equal as trees, and the corresponding labels are equal. Let $\Theta_{k,\alpha}$ denote the set of all labeled trees with $k$ nodes such that the sign label on the parent node is $\alpha$. For each $v\in V(\vartheta_l),$ we let $q(v)=\sum_{w\succeq v} p_w.$ Additionally, for a fixed $\psi \in \mathbb{T}^2$, let us define the value of each node, $\text{Val}(v,\psi)$, as 
$$\text{Val}(v,\psi) := \frac{\alpha_v}{s_v!} \lambda_{\alpha_v}^{-|p_v+1|\alpha_v} \left[\left( \prod_{j=1}^{s_v} \partial_{\alpha_{v_j}}\right)f_{\alpha_v} \right](S_0^{q(v)} \psi).$$
We call the function $F_v(\psi):=\left[\left( \prod_{j=1}^{s_v} \partial_{\alpha_{v_j}}\right)f_{\alpha_v} \right](S_0^{q(v)} \psi)$, the constant $C_v:=\alpha_v/s_v!$, the exponent $\Lambda_v:=\lambda_{a_v}^{-|p_v+1|a_v}$, and the function $g_v:=\left( \prod_{j=1}^{s_v} \partial_{\alpha_{v_j}}\right)f_{\alpha_v}$ the \textbf{F-contribution}, the \textbf{C-contribution}, the $\Lambda$-\textbf{contribution}, and the \textbf{g-contribution}, respectively, of the node $v \in V(\vartheta).$
For each labeled tree $\vartheta_l$ and a fixed $\psi \in \mathbb{T}^2,$ set
\begin{equation}
\label{tree value}
\text{Val}(\vartheta_l,\psi)=\prod_{v\in V(\vartheta_l)}\text{Val}(v,\psi).
\end{equation}
Similarly, we call the \textbf{F}/\textbf{C}/$\Lambda$/\textbf{g} \textbf{contribution of a labeled tree} to be the product of the $F/C/\Lambda/g$ contributions of all the nodes on the labeled tree. 
\end{definition}
For $\alpha \in \{+,-\}$, it was found \cite{LinearOperators} that 
\begin{equation}
\label{h^(k)}
    h_{\alpha}^{(k)}(\psi) = \sum_{\vartheta_l \in \Theta_{k, \alpha}} \text{Val}(\vartheta_l,\psi).
\end{equation}

If the point $\psi \in \mathbb{T}^2$ is fixed and clear from the context, we shall abbreviate $\text{Val}(v,\psi)$ and $\text{Val}(\vartheta_l,\psi)$ as $\text{Val}(v)$ and $\text{Val}(\vartheta)$, respectively.

\section{Method of Formal Series, Derivative Trees, and Product Trees}
\indent Now we begin to introduce the make precise of the main methodology of the paper. For a labeled tree $\vartheta$ and $v\in V(\vartheta)$, we denote $\partial_- \text{Val}(v,\psi)$ the directional derivative in the direction of $v_-$ of the function $\psi \mapsto \text{Val}(v,\psi).$ Motivated by the procedure in section 2 and 3, since  $h_{\alpha}^{(i)}(\psi) = \sum_{\vartheta \in \Theta_{i, \alpha}} \prod_{v \in V(\vartheta)} \text{Val}(v,\psi)$, we can define $\partial_- h_{\alpha}^{(k)}(\psi)$ as the following formal series:
\begin{equation}
\label{del h defination}
\partial_- h_{\alpha}^{(i)}(\psi) := \sum_{\vartheta \in \Theta_{i, \alpha}} \partial_- \left( \prod_{v \in V(\vartheta)}  \text{Val}(v,\psi) \right) = \sum_{\vartheta \in \Theta_{i, \alpha}} \left( \sum_{v \in V(\vartheta)} \partial_- \text{Val}(v,\psi) \prod_{v \in \vartheta, v \neq v_k} \text{Val}(v,\psi) \right),
\end{equation}
 The sum inside can be thought of as a tree with the label $\partial_-$ on the node $v_k$ :

\begin{definition} Let $\vartheta_l$ be a labeled tree. A $\textbf{derivative tree}$ is a tree that admits all the labels of $\vartheta_l$ along with an extra label "$\partial_-$" on an arbitrary node $v \in V(\vartheta_l).$ We call $\vartheta_l$ the $\textbf{skeleton}$ for such derivative tree and $v$ the \textbf{derivative node}. We may use the notation $\partial_v \vartheta_l$ for any derivative tree with skeleton $\theta_l$ whose derivative node is $v$, or simply $\vartheta$ if the node (or skeleton) is not specified. For $\psi \in \mathbb{T}^2$, the $\textbf{value}$ of a derivative tree $\partial_v \vartheta_l$ is defined as
$$\text{Val}(\partial_v \vartheta_l, \psi) := (\partial_- \text{Val}(v,\psi)) \prod_{v \in \vartheta, v \neq v_j} \text{Val}(v,\psi).$$
Two derivative trees are equal iff they are equal as labeled trees, and the derivative label are from the same node. Denote $\partial (\Theta_{k,\alpha})$ all the derivative trees with skeleton in $\Theta_{k,\alpha}$. We could again abbreviate $\text{Val}(\partial_v \vartheta_l, \psi)$ as $\text{Val}(\partial_v \vartheta_l)$ if the point $\psi$ is clear from the context. \\
\indent Note that if $\vartheta$ is a derivative tree, then it is also a ordinary tree, so the F/C/$\Lambda$/g contribution of $\vartheta$ or a node on $\vartheta$ is automatically defined. The \textbf{derivative}-\textbf{F}/\textbf{C}/$\Lambda$/\textbf{g} \textbf{contribution} of a node $v \in V(\vartheta)$, denoted as $D F_v,D C_v,D \Lambda_v,D g_v$, respectively, is defined as follows (let $F_v,C_v,\Lambda_v,g_v$ be the corresponding F/C/$\Lambda$/g contribution of $v$): \\
\indent (i) if $v$ is not the derivative node, then $D F_v=F_v;D C_v=C_v;D \Lambda_v=\Lambda_v;D g_v=g_v$;\\
\indent (ii) if $v$ is the derivative node, then $D F_v(\psi)=\partial_-F_v(\psi)/\lambda_-^{q(v)};DC_v=C_v;D \Lambda_v = \Lambda_v \cdot \lambda_-^{q(v)}$, and $Dg_v(\psi) = \partial_- g_v(\psi)$. \\
The \textbf{derivative}-\textbf{F}/\textbf{C}/$\Lambda$/\textbf{g} \textbf{contribution} of $\vartheta$ is the product of all the corresponding derivative contributions from each node.
\end{definition} 
\begin{figure}[h]
    \centering
    \begin{tikzpicture}
    \tikzstyle{vertex} = [circle, fill=black!255, minimum size = 1.5pt, inner sep = 1.5pt]
    \node[vertex] (v_1) at (-1, -1) {};
    \node[vertex] (v_2) at (0, -1) {};
    \node[vertex] (v_3) at (1, -1) {};
    \node[vertex] (v_4) at (2, -1) {};
    \draw (v_1) -- (v_2) -- (v_3) -- (v_4);
    \node[yshift = 10 pt] at (v_2) {$\partial_-$};
    \node[yshift = -10 pt] at (v_2) {$v$};
    \node[yshift = -10 pt] at (v_1) {$(+,2)$};
    \node[yshift = -20 pt] at (v_2) {$(-,17)$};
    \node[yshift = -10 pt] at (v_3) {$(+,0)$};
    \node[yshift = -10 pt] at (v_4) {$(-,1)$};
    \end{tikzpicture}
    \caption{A derivative tree with four vertices, with a derivative label on the node $v$}
    \label{linear derivative tree}
\end{figure}
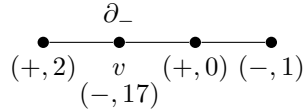 Now (\ref{del h defination}) can be re-written as 
\begin{equation}
    \partial_-h_{\alpha}^{(i)}= \sum_{\vartheta \in \partial(\Theta_{k, \alpha})} \text{Val}(\vartheta).
\end{equation}
The next goal is to find a more organized way for representing product of formal series. Note that
\begin{equation}
(\partial_- h_{\alpha_1}^{(i)})\cdot (\partial_- h_{\alpha_2}^{(j)}) =  \sum_{\vartheta_1' \in \partial(\Theta_{i, \alpha_1}), \vartheta_2' \in \partial( \Theta_{i, \alpha_2})} \text{Val}(\vartheta_1') \times \text{Val}(\vartheta_1'),
\end{equation}
which is a formal sum over the values of all the "product of trees":

\begin{definition} A $\textbf{product tree}$ $p$ is an element in Cartesian products of the form $\prod_ {i=0}^{s}\partial (\Theta_{k_i,\alpha_i})$, where $k \geq 0$. We use $s(p), n(p)$ and $p(i)$ to represent the number of factors $-$ 1, the total number of nodes across all factors, and the derivative tree with index $i$ in $p$.($i$ always begin with 0).  If $s(p)=0,$ we sometimes make the obvious identification with the derivative tree it represents. The value of $p$ is defined as
$$\text{Val}(p) := \prod_{i=0}^{s(p)} \text{Val}(p(i)).$$
Therefore, we obtain a more compact notation
\begin{equation}
    (\partial_- h_{\alpha_0}^{(k_0)})\cdots(\partial_- h_{\alpha_s}^{(k_s)})= \sum_{p \in \prod_ {i=0}^{s}\partial (\Theta_{k_i,\alpha_i})} \text{Val}(p)
\end{equation}
for product of formal series. A product tree whose components are $\vartheta_1,...,\vartheta_k$ should be denoted as $\vartheta_1 \times \cdots \times \vartheta_k$ instead of $(\vartheta_1 ,\cdots , \vartheta_k)$ to avoid confusion with a single tree. We use $V(p)$ to denote the set of all nodes across all factors of $p$. Two product trees are equal iff they have the same number of components and the derivative trees on each components are equal. \\
\indent If $p$ is a product tree, the \textbf{(derivative-)}\textbf{F}/\textbf{C}/$\Lambda$/\textbf{g} \textbf{contribution} of a node $v \in V(\vartheta)$ is the corresponding (derivative-) contribution of the node from the corresponding tree factor. The \textbf{(derivative-)}\textbf{F}/\textbf{C}/$\Lambda$/\textbf{g} \textbf{contribution} of $\vartheta$ the product of all the (derivative-) contribution from each factor tree.

\end{definition}
Now, motivated by equation (\ref{q_n(t)}) and the procedure for the second order, we want to show that after "cancellation", the formal series $q_n(0)$ defined as
\begin{equation}
\label{q_n(0)}
     q_n(0) := \partial_-h_{+}^{(n)} + \sum_{k = 1}^{n - 1}\partial_-h_{+}^{(k)} \left(\sum_{s = 0}^{n - k}(-1)^s \sum_{\substack{m_1 + ... + m_s = n - k \\ m_1,...,m_s \geq 0}} \partial_-h_-^{(m_1)} \cdot ... \cdot \partial_-h_-^{(m_s)} \right)
 \end{equation}
converges. The following lemma is helpful in understanding the structure of $(\ref{q_n(0)})$:

\begin{lemma} 
\label{alternate form of q_n(0) Lemma}
Let $\mathcal{P}_{k, +}$ be the collection of all product trees $p$ such that $n(p)=k$ and the label on the top node of $p(i)$ is $+$ if $i=0$, and $-$ otherwise. For all $n \geq 1,$ we have, as formal series without cancellations,
\begin{equation}
\label{tree representation for q_0}
 q_n(0)=\sum_{p \in \mathcal{P}_{n,+}} (-1)^{s(p)}\text{Val}(p).   
\end{equation}

\end{lemma}
\begin{proof}
    By (\ref{q_n(0)}),
    \begin{align}
      q_n(0) & =\sum_{\substack{m_1+...+m_k=n \\ m_1,...,m_s \geq 0}} (-1)^{k-1} (\partial_- h_+^{(m_1)})( \partial_- h_-^{(m_2)})\cdots (\partial_- h_-^{(m_s)})  \notag \\
      & = \sum_{\substack{m_1+...+m_k=n \\ m_1,...,m_s \geq 0}}  (-1)^{k-1} \sum_{\substack{p \in \prod_ {i=0}^{s}\partial (\Theta_{m_i,\alpha_i}) \\ \alpha_0=+;\alpha_i=- \forall i \geq 1}} \text{Val}(p) \notag \\
      & = \sum_{\substack{m_1+...+m_k=n \\ m_1,...,m_s \geq 0}}  \sum_{\substack{p \in \prod_ {i=0}^{s}\partial (\Theta_{m_i,\alpha_i}) \\ \alpha_0=+;\alpha_i=- \forall i \geq 1}} (-1)^{s(p)} \text{Val}(p) \notag \\
      & = \sum_{p \in \mathcal{P}_{n,+}} (-1)^{s(p)}\text{Val}(p). \notag
    \end{align}
\end{proof}

\indent At this point we should be more careful with canceling formal series, as in general it is not a well defined procedure. For a simplest example, consider the two formal series $1+2+3+...$ and $2+3+4+...$. How could we define their difference? Some obvious candidates include just the number $1$, the formal series $(-1)+(-1)+(-1)+...$, or the series $1+2+1+2+3+4+...$ (subtract the second series from the first one, starting at the third position), etc. Some of these are finite, and the ones that does not converge differs dramatically in its form. \\
\indent Coming back to our case, just like what we did in (\ref{del h^2 +})-(\ref{del h+ * del h-})-(\ref{break into two parts}), we will treat all the sums in the series to be formal, with one exception that if there are two terms $C_1\cdot \Lambda_1 \cdot F_1(\psi)$ and $C_2\cdot \Lambda_2 \cdot F_2(\psi)$ satisfying $F_1(\psi)=F_2(\psi)$ as functions, then they could be combined to form $(C_1\cdot \Lambda_2+C_1\cdot \Lambda_2) \cdot F_1(\psi)$. If these coefficients sums up to 0, we will drop the term. This procedure will always give a unique series (up to rearrangements and combining terms alike), provided that for each $F(\psi),$ the number of terms of the form $C_k \cdot \Lambda_k \cdot F(\psi)$ is finite: 
\begin{proposition}
\label{series cancelation lemma}
   For each $p \in \mathcal{P}_{n,+},$ only finitely many $q \in \mathcal{P}_{n,+}$ has a derivative-F contribution that's equal to $p$'s. Furthermore, after all possible cancellations, $q_n(0)$ converges absolutely in the ordinary sense.
\end{proposition}
This will be a tedious proposition to prove, we still need more machinery.

\begin{definition}
\label{equivlant trees}
    A \textbf{half-labeled derivative tree} is a tree such that on each of its node there is a sign label, and there is a $\partial_-$ on exactly one of the nodes. (In particular, there are no integer labels). Two derivative trees $\vartheta_1$ and $\vartheta_2$ are \textbf{equivalent} if they are equal as half-labeled derivative trees. Two product trees $p_1,p_2$ are \textbf{equivalent} if $s(p_1)=s(p_2),$ and $p_1(i)$ is equivalent to $p_2(i)$ for every $0 \leq i \leq s(p_1).$
\end{definition}

Clearly this defines an equivalence class on the set of product trees, and we the notation $[p]$ to denote the equivalence class given by $p$. set $H \partial (\Theta_{k,\alpha})$ to be the set that contains all equivalence classes of the form $[\vartheta]$, where $\vartheta \in \partial (\Theta_{k,\alpha})$. At this point we have enough tools to prove the first assertion of proposition \ref{series cancelation lemma}:
\begin{lemma}
Let $n \geq 1$ be an integer.\\
    \indent (a) If $p$ is a product tree with $|V(p)|=n$, then the cardinality of the set of functions $\{F_q(\psi):q \in [p], F_q(\psi)=F_p(\psi)\}$ is finite. (Here, $F_p(\psi)$ is the derivative-F contribution of $p$). \\    
    \indent (b) For each $p \in \mathcal{P}_{n,+},$ only finitely many $q \in \mathcal{P}_{n,+}$ has a derivative-F contribution that's equal to $p$'s.
\end{lemma}
\begin{proof}
    For (a), let $v_1,...,v_n$ be an enumeration of elements in $V(p_1)=V(p_2)$ such that nodes from the same tree are consecutive, and if $v_i,v_j$ are on the same tree and $i<j$, then the number of ancestors of $v_i$ is no more than the number of ancestors of $v_j$. (By ancestor of a node $v$, we mean any node along the shortest path from $v$ to the top node). The derivative-F contribution $F(\psi)$ of  any trees in $[p]$ has the form
    \[g_1(S_0^{q(v_1)} \psi)g_2(S_0^{q(v_2)} \psi) \cdots g_n(S_0^{q(v_n)}\psi).\]
    Where $g_i$ are the derivative-$g$ contribution of $v_i \in V(p).$ We will abuse notation and also use $v_i$ to denote the integer label on the node $v_i$. First assume all the $g_i$'s are different. In this case it suffices to show the correspondence $(v_1,v_2,...,v_n) \rightarrow (q(v_1),q(v_2),...,q(v_n))$ is injective as a function from $\mathbb{Z}^n$ to $\mathbb{Z}^n.$ This follows because the matrix of this function is an upper-triangular matrix with all 1's on the diagonal. If some of the $g_i$'s coincides, then at most $n! < \infty$ tuples of $(v_1,v_2,...,v_n)$ gives the same function $F(\psi)$. Thus part (a) Follows.  (b) Follows directly from (a) and Lemma \ref{alternate form of q_n(0) Lemma} because the number of equivalence classes $[p]$ in $\mathcal{P}_{n,+}$ is finite.
\end{proof}
Thus we can now combine terms alike in the series $q_n(0)$ knowing that there could only be a unique reduced form.

\begin{definition}
    Let $\vartheta$ be a derivative tree. Let $\text{Perm}(\vartheta)$ be the set of all derivative trees obtained by permuting the elements in one or more nodes of $\vartheta$. (Recall that a node is a tuple of trees). If $p=\vartheta_1 \times \cdots \times \vartheta_k$ is a product tree, then $\text{Perm}(p) :=\text{Perm}(\vartheta_1) \times \cdots \times \text{Perm}(\vartheta_k).$ If $C$ is a set of product trees, define $\text{Perm}(C)=\cup_{p \in C} \text{Perm}(p).$
\end{definition}
For example, one clearly has $|\text{Perm}(\vartheta)|=\prod_{v \in V(\vartheta)}s_v!.$ 

\begin{definition}
\label{breaking class}
Let $\vartheta$ be a derivative tree with derivative label on $v\in V(\vartheta).$ The $\textbf{main stem}$ $M\vartheta$ of $\vartheta$ is the order tuple of nodes $(v_1,v_2,...,v_k)$ where $v_1$ is the top node of $\vartheta$, $v_k=v$, and $v_i$ is the parent of $v_{i+1}$ for  $i=1,...,k-1.$  The $\textbf{relatives}$ of a node $v_i$ on $M \vartheta$, denoted $R(v_i)$, consists of all nodes $w$ such that $v_i  \succeq  w$ and $v_j \nsucceq
 w$ for all $j >i.$ A tuple of nodes $(w_1,w_2,...,w_j)$ on $M \vartheta$ such that $j<k, w_1 \neq v_1$, $w_1 \succeq \cdots \succeq w_j$, and $\alpha_{w_i}=-\forall i \leq j$ is called a $\textbf{cut}.$ The empty tuple is allowed to be a cut. The $\textbf{breaking}$ of $\vartheta$ according to the cut $(w_1,w_2,...,w_j)$ is the product tree $\vartheta_0 \times \vartheta_1 \times \cdots \times \vartheta_j$ where $\vartheta_i$ is obtained by deleting the node $w_{i+1}$ from the tree $w_i$ while maintaining the order and labels of all the other nodes, and attach a $\partial_-$ on the node where it is broken. Here we call $w_0=\vartheta$ and $w_{j+1}=$ null. The breaking of $\vartheta$ according to the empty tuple is $\vartheta$ itself (identified with the 1-element product). The $\textbf{breaking class}$ of $\vartheta,$ $\text{brk}(\vartheta)$, consists of breaking of $\vartheta$ according to all possible cuts.
\end{definition}

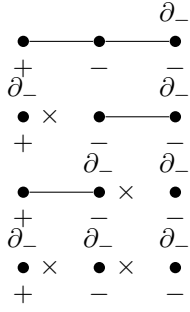
\begin{figure}[h]
    \centering
    \begin{tikzpicture}
    \tikzstyle{vertex} = [circle, fill=black!255, minimum size = 1.5pt, inner sep = 1.5pt]
    \node[vertex] (v_1) at (0, 0) {};
    \node[yshift = -10 pt] at (v_1) {+};
    \node[vertex] (v_2) at (1, 0) {};
    \node[yshift = -10 pt] at (v_2) {$-$};
    \node[vertex] (v_3) at (2, 0) {};
    \node[yshift = 10 pt] at (v_3) {$\partial_-$};
    \node[yshift = -10 pt] at (v_3) {$-$};
    \draw (v_1) -- (v_2) -- (v_3);
    \node[vertex] (v_4) at (0, -1) {};
    \node[yshift = -10 pt] at (v_4) {+};
    \node[yshift = 10 pt] at (v_4) {$\partial_-$};
    \node[xshift = 10 pt] at (v_4) {$\times$};
    \node[vertex] (v_5) at (1, -1) {};
    \node[yshift = -10 pt] at (v_5) {$-$};
    \node[vertex] (v_6) at (2, -1) {};
    \node[yshift = 10 pt] at (v_6) {$\partial_-$};
    \node[yshift = -10 pt] at (v_6) {$-$};
    \draw (v_5) -- (v_6);
    \node[vertex] (v_7) at (0, -2) {};
    \node[yshift = -10 pt] at (v_7) {+};
    \node[vertex] (v_8) at (1, -2) {};
    \node[yshift = -10 pt] at (v_8) {$-$};
    \node[yshift = 10 pt] at (v_8) {$\partial_-$};
    \node[xshift = 10 pt] at (v_8) {$\times$};
    \node[vertex] (v_9) at (2, -2) {};
    \node[yshift = 10 pt] at (v_9) {$\partial_-$};
    \node[yshift = -10 pt] at (v_9) {$-$};
    \draw (v_7) -- (v_8);
    \node[vertex] (v_10) at (0, -3) {};
    \node[yshift = -10 pt] at (v_10) {+};
    \node[yshift = 10 pt] at (v_10) {$\partial_-$};
    \node[xshift = 10 pt] at (v_10) {$\times$};
    \node[vertex] (v_11) at (1, -3) {};
    \node[yshift = -10 pt] at (v_11) {$-$};
    \node[yshift = 10 pt] at (v_11) {$\partial_-$};
    \node[xshift = 10 pt] at (v_11) {$\times$};
    \node[vertex] (v_12) at (2, -3) {};
    \node[yshift = 10 pt] at (v_12) {$\partial_-$};
    \node[yshift = -10 pt] at (v_12) {$-$};
    \end{tikzpicture}
    \caption{Product Trees in one breaking class. Number labels are not shown.} 
\end{figure}
In other words, the breaking class of a tree $\vartheta$ consist of all the product trees formed by "separating" the nodes in the main stem where the there is a minus label, and adding an extra "$\partial_-$" on where it detached from. Thus, if the number of "$-$" labels along the main stem of $\vartheta$ is $k$, then $\text{brk} (\vartheta)$ contains of $2^k$ elements. If $C$ is a set of derivative trees, define $\text{brk}(C)=\cup_{p\in C} \text{brk}(C).$ \\ 
\indent Denote $\mathcal{C}(M \vartheta)$ the set of all possible cuts of $M \vartheta$.  Note that the information of a cut $c\in \mathcal{C}(M \vartheta)$ can be encoded as a fixed length tuple of 0 and 1's, where the length of the tuple is the number of nodes in $M \vartheta$ that has a minus label. For example, if $(v_{i_1},v_{i_2},v_{i_3})$ are all the nodes in $M \vartheta$ that has a minus label, then the tuple $(0,1,1)$ denotes the cut $(v_{i_2},v_{i_3})$, and we shall call this tuple the \textbf{string representation} of the corresponding cut.

\begin{lemma}
    Let $p$ be a product tree with $n$ nodes such that $s(p)=0$ (thus identified with a derivative tree). \\
    \indent (a) If $p_1,p_2 \in \text{Perm}([p])$, then $\text{brk}(\text{Perm}([p_1]))=\text{brk}(\text{Perm}([p_2]))$; \\
    \indent (b) If $p_1 \in \text{Perm}([p])$ but $p_2 \notin \text{Perm}([p])$, then $\text{brk}(\text{Perm}([p_1])) \cap \text{brk}(\text{Perm}([p_2])) = \O$; \\
    \indent (c) $\cup_{\vartheta \in \partial (\Theta_{n,+})}\text{brk}(\text{Perm}([\vartheta]))$ = $\mathcal{P}_{n,+}$.
\end{lemma}
\begin{proof}
    For (a),  note that $p_1,p_2 \in \text{Perm}([p])$ is also an equivalence relation, so $\text{Perm}([p])=\text{Perm}([p_1])=\text{Perm}([p_2])$ and the result follows. For (b) we prove the contrapositive. Assume $p_1\in \text{Perm}([p])$ and there is $p^*=\partial_{v_0}\vartheta_0\times \cdots \times \partial_{v_s}\vartheta_s$ in the intersection. Note that the breaking procedure described in definition \ref{breaking class} is invertible up to permutation, so we could reassemble $p^*$ and get a derivative tree $p'$ satisfying $p' \in \text{Perm}([p_1])$ and $p' \in \text{Perm}([p_2])$. Thus $\text{Perm}([p_1])=\text{Perm}([p_2])$ so $p_2 \in \text{Perm}([p]).$ For part (c), the "$ \subseteq$" direction is immediate. For the reverse, just assemble any product trees in $\mathcal{P}_{n,+}$ as like above, and it will be the tree $\vartheta$.
\end{proof}

This lemma shows that sets of the form $\text{brk}(\text{Perm}([\vartheta]))$ gives a partition of $\mathcal{P}_{n,+}$, and the map $\text{Perm}([\vartheta])\mapsto \text{brk}(\text{Perm}([\vartheta]))$ is a bijection. Since number of equivalence classes $\text{Perm}([\vartheta])$ in $\partial(\Theta_{n,+})$ is finite, along with lemma \ref{alternate form of q_n(0) Lemma}, we just need the following to prove proposition \ref{series cancelation lemma}:
\begin{lemma} 
\label{hard lemma}
Let $\vartheta \in \partial (\Theta_{n,+}).$ Then the formal series
\begin{equation}
\label{series grouped into breaking class}
  \sum_{p \in \text{brk}(\text{Perm}([\vartheta]))} (-1)^{s(p)}\text{Val}(p)   
\end{equation}
converges absolutely after all possible cancellations.
\end{lemma}
\begin{proof}
    Let's begin with the special case where $\vartheta$ is linear and $\partial_-$ is labeled on the smallest node (with respect to $\succeq$). Let $M\vartheta=(v_1,...,v_k)$ be the main stem of $\vartheta$, which in this case are exactly all of the nodes in $\vartheta$ (Even though $k=n$, the reason we are using $k$ instead of $n$ will be clear later). Furthermore, there is only one permutation for any linear derivative trees or product of linear derivative trees (i.e., itself). Now, (\ref{series grouped into breaking class}) could be regrouped according to the type of the cut :
    \begin{equation}
    \label{group according to cuts}
        \sum_{c \in \mathcal{C}(M \vartheta)} \sum_{\substack{p \in \text{brk}(\text{Perm}([\vartheta]) \\ c \rightarrow p}} (-1)^{|c|} \text{Val}(p),
    \end{equation}
    where $\mathcal{C}(M \vartheta)$ denotes all cuts on $M \vartheta,$ $c\rightarrow p$ represents the relation "$p$ is a product tree formed from the cut $c$", and $|c|$ denotes the size of the cut (number of 1's in the string representation). Observe for any $c \in \mathcal{C}(M \vartheta)$, the inside sum of (\ref{group according to cuts}) could be re-written as
    \begin{equation}
       \mathcal{S}_c=(-1)^{|c|} \sum_{a_1,...,a_k} \alpha \Lambda (A_1^c)_1(A_2^c)_2\cdots(A_j^c)_k,
    \end{equation}
    where: \\
    \indent (i) The range of $\alpha_i$ is in $\mathbb{Z}_{\alpha_{v_i}}$; \\
   \indent (ii) $\alpha$ is the product of all the sign labels on $V(\vartheta)$;   \\
   \indent (iii) the numbers $A^c_i$ are defined as $a_{j}+a_{j+1}+...+a_{i}$ where $v_j \in c$ but $v_{j+1},v_{j+2},...,v_i \notin c$. If there is no $j <i$ such that $v_j \in c$, then $A^c_i:=a_1+...+a_i$.\\
   \indent (iv) the parenthesis notation is a short hand way of denoting 
    \[(x)_i=g_i(S_0^{x} \psi),\]
    \indent where $g_i$ is the derivative-$g$ contribution of $v_i$; and \\
   \indent (v) $\Lambda$ is a function in $a_1,....,a_n$ that is summable in $a_i$ if $\alpha_{v_i}=+$, and not summable otherwise. \\
   It's easy to see that $\alpha$ and $\Lambda$ are invariant under any $c \in \mathcal{C}(M \vartheta).$ \\
   \indent Now we describe the strategy for cancellation. Let $v_{m_1},...,v_{m_s}$ be all the nodes on $M \vartheta$ whose sign labels are negative. For a tuple $t$ of 0's and 1's, we call the \textbf{conjugate} of $t$, denoted $\overline{t}$ to be the tuple obtained by switching the first element in $t$ from 1 to 0 (or from 0 to 1). For a cut $c \in \mathcal{C}^{(0)}(M\vartheta):=\mathcal{C}(M\vartheta)$, we defined its $\textbf{conjugate}$ $ \overline{c}$ to be the cut whose string representation is the conjugate of that of $c$. We now run the following algorithm with $i=1.$  \\
   \indent Now the set $\mathcal{C}^{(i-1)}(M\vartheta)$ has already been defined. Let $\mathcal{C}^{(i)}(M\vartheta)$ be the set of two-element sets of the form $\{c^{i-1},\overline{c^{i-1}}\}$ for $c^{i-1} \in \mathcal{C}^{(i-1)}(M\vartheta)$. Note that $|\mathcal{C}^{(i)}(M\vartheta)|=2^{s-i}$. Then for $c^i=\{c^{i-1},\overline{c^{i-1}}\} \in \mathcal{C}^{(i)}(M\vartheta)$, we set $\mathcal{S}_{c^i}=\mathcal{S}_{c^{i-1}}+\mathcal{S}_{\overline{c^{i-1}}}$, and we will show that a significant amount of cancellation occurs within $\mathcal{S}_{c^i}$. An element $\{c^{i-1},\overline{c^{i-1}}\}$ in $\mathcal{C}^{(i)}(M\vartheta)$ can be described as a $s-i$ tuple of 0's and 1's, which is formed from deleting the first element of the string representation of $c^{i-1}$ (or $\overline{c^{i-1}}$), which is a $s-i+1$ tuple. We call this tuple the \textbf{string representation} of $\{c^{i-1}, \overline{c^{i-1}}\}$. Two elements in $\mathcal{C}^{i}(M\vartheta)$ are \textbf{conjugate} to each other if their string representations are. Now repeat the procedure in this paragraph with $i=i+1.$ Stop repeating until we reach a single series $\mathcal{S}_{c^s}$. We show $\mathcal{S}_{c^s}$ converges.\\
   \indent We now follow the procedure. Choose $c, \overline{c} \in \mathcal{C}^{(0)}(M \vartheta)$. Assume $c$ is the cut whose whose string representation starts at 0. One has, by a change of index,
   \begin{align}
       \mathcal{S}_c = & (-1)^{|c|}\sum_{a_1,...,a_{(m_1-1)}} \sum_{a_{m_1}=-1}^{-a_1-...-a_{(m_1-1)}} \sum_{a_{({m_1}+1)},...,a_k} \alpha \Lambda (A_1^c)_1(A_2^c)_2\cdots(A_j^c)_k + \notag \\
        & (-1)^{|c|}\sum_{a_1,...,a_{m_1-1}} \quad \sum_{a_{m_1}=(-a_1-...-a_{(m_1-1)})-1}^{-\infty} \quad \sum_{a_{({m_1}+1)},...,a_k} \alpha \Lambda (A_1^c)_1(A_2^c)_2\cdots(A_j^c)_k  \notag \\
        = &  (-1)^{|c|}\sum_{a_1,...,a_{(m_1-1)}} \sum_{a_{m_1}=-1}^{-a_1-...-a_{(m_1-1)}} \sum_{a_{({m_1}+1)},...,a_k} \alpha \Lambda (A_1^c)_1(A_2^c)_2\cdots(A_j^c)_k + \notag \\
        & (-1)^{|c|}\sum_{a_1,...,a_{(m_1-1)}} \quad \sum_{a_{m_1}=-1}^{-\infty} \quad \sum_{a_{({m_1}+1)},...,a_k} \alpha \Lambda (A_1^{\overline{c}})_1(A_2^{\overline{c}})_2\cdots(A_j^{\overline{c}})_k  \notag \\
        = & (-1)^{|c|}\sum_{a_1,...,a_{(m_1-1)}} \sum_{a_{m_1}=-1}^{-a_1-...-a_{(m_1-1)}} \sum_{a_{({m_1}+1)},...,a_k} \alpha \Lambda (A_1^c)_1(A_2^c)_2\cdots(A_j^c)_k + (- \mathcal{S}_{\overline{c}}).\notag
   \end{align}
  Here the change of index is valid as $\Lambda$ does not depend on $a_{m_1}$. Now
  \begin{equation}
  \label{first cancelation}
      \mathcal{S}_{\{c,\overline{c}\}}= (-1)^{|c|}\sum_{a_1,...,a_{m_1-1}} \sum_{a_{m_1}=-1}^{-a_1-...-a_{m_1-1}} \sum_{a_{({m_1}+1)},...,a_k} \alpha \Lambda (A_1^c)_1(A_2^c)_2\cdots(A_j^c)_k .
  \end{equation}
   Note that after the above cancellation, the series  $\mathcal{S}_{\{c, \overline{c}\}}$ is somewhat "better" than 
   $\mathcal{S}_c$ and $\mathcal{S}_{\overline{c}}$ in the sense that we remove the divergent variable $a_{m_1}$. Furthermore, we now went from $2^s$ series of the form $\{\mathcal{S}_c:c\in \mathcal{C}^{(0)}(M \vartheta)\}$ down to $2^{s-1}$ series of the form $\{ \mathcal{S}_{ \{c,\overline{c}\} }:c\in \mathcal{C}^{(0)}(M \vartheta)\}=\{ \mathcal{S}_{c^1}:c^1\in \mathcal{C}^{(1)}(M \vartheta)\}$. We re-denote (\ref{first cancelation}) as, for every $s^1 \in \mathcal{C}^{(1)}(M \vartheta), $
  \begin{equation}
  \label{first cancel, simplified}
      \mathcal{S}_{c^1}= (-1)^{|c^1|}\sum_{a_1,...,a_k} \alpha \Lambda (A_1^{c^1})_1(A_2^{c^1})_2\cdots(A_j^{c^1})_k,
  \end{equation}
    where $|c^1|$ denotes the number of 1's in the string representation of $c^1$, and set $A^{c^1}_i:=A^{c}_i$, where $c$ is the element in $c^1$ whose string representation begin with 0. Each $a_i$'s are understood to have the range in (\ref{first cancelation}). Observe that since if $c^1=\{c, \overline{c}\}$, and $c$ starts with 0, then $|c|=|c^1|$. Now we perform the cancellation again, observe that in (\ref{first cancel, simplified}), since $-a_1-...-a_{(a_2-1)} \leq -1$ for every $a_1,...,a_{(a_2-1)}$ in its range, one has, by essentially the same change of index as above,
    \begin{align}
        \mathcal{S}_{c^1} = & (-1)^{|c^1|}  \sum_{a_1,...,a_{(m_2-1)}} \quad 
        \sum_{a_{m_2}=-1}^{-a_1-...-a_{(m_2-1)}} \sum_{a_{({m_2}+1)},...,a_k} \alpha \Lambda (A_1^{c^1})_1(A_2^{c^1})_2\cdots(A_j^{c^1})_k + \notag \\
        & (-1)^{|c^1|}  \sum_{a_1,...,a_{(m_2-1)}} \quad 
        \sum_{a_{m_2}=-a_1-...-a_{(m_2-1)}-1}^{-\infty}\quad  \sum_{a_{({m_2}+1)},...,a_k} \alpha \Lambda (A_1^{c^1})_1(A_2^{c^1})_2\cdots(A_j^{c^1})_k \notag \\
        = & (-1)^{|c^1|}  \sum_{a_1,...,a_{(m_2-1)}} \quad 
        \sum_{a_{m_2}=-1}^{-a_1-...-a_{(m_2-1)}} \sum_{a_{({m_2}+1)},...,a_k} \alpha \Lambda (A_1^{c^1})_1(A_2^{c^1})_2\cdots(A_j^{c^1})_k + \notag \\
        &   (-1)^{|c^1|}\sum_{a_1,...,a_{(m_2-1)}} \quad \sum_{a_{m_2}=-1}^{-\infty} \quad \sum_{a_{({m_2}+1)},...,a_k} \alpha \Lambda (A_1^{\overline{c^1}})_1(A_2^{\overline{c^1}})_2\cdots(A_j^{\overline{c^1}})_k  \notag \\
        = & (-1)^{|c^1|}\sum_{a_1,...,a_{(m_2-1)}} \sum_{a_{m_2}=-1}^{-a_1-...-a_{(m_2-1)}} \sum_{a_{({m_2}+1)},...,a_k} \alpha \Lambda (A_1^{c^1})_1(A_2^{c^1})_2\cdots(A_j^{c^1})_k + (- \mathcal{S}_{\overline{c^1}}).\notag
    \end{align}
    Still, we are allowed to do these since $\Lambda$ does not depends on $a_{m_2}$. We obtain
    \[\mathcal{S}_{\{c^1,\overline{c^1}\}}= (-1)^{|c^1|}\sum_{a_1,...,a_{m_2-1}} \sum_{a_{m_2}=-1}^{-a_1-...-a_{m_2-1}} \sum_{a_{({m_2}+1)},...,a_k} \alpha \Lambda (A_1^{c^1})_1(A_2^{c^1})_2\cdots(A_j^{c^1})_k .\]
    (Compare (\ref{first cancelation})). This procedure will continue as described previously, until we left with a single series
    \begin{equation}
    \label{fully canceld}
        \mathcal{S}_{c^s}= \sum_{a_1,...,a_k} \alpha \Lambda(a_1)_1(a_1+a_2)_2\cdots (a_1+a_2+...+a_k)_k,
    \end{equation}
    where the range of $a_j$ is $[-a_1-...-a_{j-1},-1] \cap \mathbb{Z}$ if $\alpha_{v_{a_j}}=-$, and is $\mathbb{Z}^+$ otherwise. Observe that since $(a_1)_1(a_1+a_2)_2\cdots (a_1+a_2+...+a_k)_k$ is bounded by some constant $M$ for all inputs of $a_j,$
    we get (assume $r$ is the number of positive labels on the main stem)
    \begin{align}
    \label{S_c bound}
         |\mathcal{S}_{c^s}| & \leq \sum_{n_1=0}^{\infty} \sum_{n_2=0}^{\infty} \cdots \sum_{n_r=0}^{\infty} M n_1 \cdot (n_1+n_2) \cdots (n_1+n_2+...+n_r) \lambda_+^{-2n_1} \cdots \lambda_+^{-2n_r} \notag \\
         & \leq M \sum_{n_1=0}^{\infty} \sum_{n_2=0}^{\infty} \cdots \sum_{n_r=0}^{\infty} [(n_1+1)(n_2+1)\cdots (n_r+1)]^r \lambda_+^{-2n_r} \cdots \lambda_+^{-2n_r} \notag \\
         & = M \left [  \sum_{n_1=0}^{\infty} (n_1+1)^r \lambda_+^{-2n_1} \right] \cdots \left [  \sum_{n_r=0}^{\infty} (n_r+1)^r \lambda_+^{-2n_r} \right] \notag \\
         & = M \frac{\lambda_+^{-2r} [A_r(\lambda_+^{-2})]^r}{(1-\lambda_+^{-2})^{r(r+1)}},
    \end{align}
    The last equality follows from the explicit formula for the polylogarithm $\text{Li}_s(z)$ at $s=-r$ and $z=\lambda_+^{-2}.$ The function $A_r(x)$ is the monic polynomial of degree $r$ defined as $\sum_{k=0}^{r} A(r,k)x^k$, where $A(r,k)$ are Eulerian Numbers  \cite{Wood1992}. Thus (\ref{fully canceld}) converges absolutely in the ordinary sense. \\
    \indent The proof for the general case is almost exactly the same as the proof for the linear case. The exact same proof could be followed again with the understanding that $\vartheta$ is now a general derivative tree (the notations are designed to fit the new context as well), except that the parenthesis notation $(x)_i$ now refer to the formal series (assuming $v_{i_1},...,v_{i_w}$ are all the elements in $R(v_i)$ excluding $v_i$)
    \[(x)_i:=\sum_{b_1,...,b_w} \Lambda_{R(v_i)} g_i(S_0^{x} \psi) g_{i_1}(S_0^{x+B_1} \psi) \cdots g_{i_w}(S_0^{x+B_w} \psi) \]
     where $b_d \in \mathbb{Z}_{\alpha_{v_{i_d}}}$, and they are understood to be the integer labels on $v_{i,d}$; $\Lambda_{R(v_i)}$ is the product of all the derivative-$\Lambda$ contributions of $v\in R(v_i), v \neq v_i;$ $g_i$ is the derivative-$g$ contribution on $v_i$; $g_{i_d}$ is the derivative-$g$ contribution on $v_{i_d}$;  and $B_d$ is defined to be the sum of all the $b_e$'s such that $v_{i_e}$ is an ancestor (or itself) of $v_{i_d}$, and $v_{i_e} \in R(v_i) -\{v_i\}.$  Furthermore, equation (\ref{S_c bound}) should sum over the nodes that's not on the main stem as well (which is also convergent).\\
     \indent Finally, note that equation (\ref{first cancelation}) stays the same without the constant $C$ (which consists of the product of all the reciprocals of factorials) thanks to our use of the Perm operation. This concludes the proof.
\end{proof}
The proposition suggest that we can define
$\text{Val}[q_n(0)]$, the value of $q_n(0)$ after all the cancellations, to be the (finite) sum of all series of the form (\ref{fully canceld}). It is a finite sum since each such series corresponds to an element of $H\partial(\Theta_{k,\alpha})$, which is a finite set. \\
\indent In notation, if for $E \in H\partial(\Theta_{k,\alpha})$, we define $\Tilde{E}$ to be the subset of $E$ containing all derivative trees $\vartheta$ such that if the sign label $\alpha_v$ on $v \in V(\vartheta)$ is $-,$ then the integer label $p_v$ is in the range $[-p_{v_1}-...-p_{v_j},-1]\cap \mathbb{Z}$, where $v_1,...,v_j$ are all the ancestors of the node $v$, and $p_{v_i}$ are the corresponding integer labels. Now if we set

$$\text{Val}(E)=\sum_{\vartheta \in \Tilde{E}} \text{Val}(\vartheta),$$
then,
\begin{equation}
    \text{Val}[q_n(0)] = \sum_{E \in H\partial(\Theta_{n,+})} \text{Val}(E).
\end{equation}

So far, all the work we've done are at a formal level, and we still have to make everything work at a rigours level. As we will see, the rigours procedure is almost exactly the same as our formal series procedure. We will develop this in the next section. 

\section{From Formal Series to A Rigours Procedure}
Lets begin with proving (iii) of Remark \ref{proof steps}. We will use the method of Fourier Transform. For a function $g$ defined on $\mathbb{T}^2=\mathbb{R}^2/\mathbb{Z}^2$, we set
\[\hat{g}(n)=\int_{\mathbb{T}^2}g(x)e^{-in \cdot x}dx.\]
Some elementary properties are summarized below. See \cite{Folland} for detailed proof.
\begin{proposition}
    Let $f:\mathbb{T}^2 \rightarrow \mathbb{R}$. Then: \\
    \indent (a) $\widehat{\partial_v f}(n)=i(v\cdot n) \hat{f}(n)$; \\
    \indent (b) If $T \in GL(n,\mathbb{Z})$ and $\det T=\pm1$, then $f \circ T$ is well defined on $\mathbb{T}^2,$ and $\widehat{f \circ T}(n)=\hat{f}(S^{-a}n)$. Here, $S=(T^*)^{-1}.$ 
\end{proposition}
With these tools, we could prove the following:
\begin{proposition} 
\label{converge prop}
There exist $D>0$ such that the series $\sum_{k=0}^{\infty} \epsilon ^k \text{Val} [q_k(0)] $ converges absolutely for $|\epsilon|<D.$
\end{proposition}
\begin{proof}
    Fix a derivative tree $\vartheta$ and $v,w \in V(\vartheta)$. Assume $v$ is not on the main stem of $\vartheta$, but $w$ is. Use $\text{Val}_v(\psi),\text{Val}_w(\psi)$ to denote $\text{Val}(v, \psi), \text{Val}(w, \psi)$. Since
    \[\text{Val}_v (\psi)=\frac{\alpha_v}{s_v!}\lambda_{\alpha_v}^{-|p_v+1|\alpha_v} \left( \prod_{j=1}^{s_v} \partial_{a_{v_j}}\right)f_{\alpha_v}(S_0^{p(v)} \psi);\]
    \[\text{Val}_w (\psi)=\frac{\alpha_w}{s_w!}\lambda_{\alpha_v}^{-|p_v+1|\alpha_v+p_w} \left(\partial_- \prod_{j=1}^{s_v} \partial_{a_{v_j}}\right)f_{\alpha_v}(S_0^{p(w)} \psi).\]
    Taking the Fourier Transform, we have, by the above proposition, 
    \[\widehat{\text{Val}_v}(n)=\frac{\alpha_v}{s_v!}\lambda_{\alpha_v}^{-|p_v+1|\alpha_v} \prod_{j=1}^{s_v} \left (i v_{\alpha_{v_j}} \cdot S_0^{-p(v)} n \right ) \hat{f}_{\alpha_v}(S_0^{-p(v)}n);\]
     \[\widehat{\text{Val}_w}(n)=\frac{\alpha_w}{s_w!}\lambda_{\alpha_v}^{-|p_v+1|\alpha_v+p_w} \left (i v_{-} \cdot S_0^{-p(v)} n \right ) \prod_{j=1}^{s_v} \left (i v_{\alpha_{v_j}} \cdot S_0^{-p(v)} n \right ) \hat{f}_{\alpha_v}(S_0^{-p(v)}n ).\]
     Therefore, one now has $\text{Val}_v(\psi)=\sum_{n \in \mathbb{Z}^2}e^{in\cdot \psi}\widehat{\text{Val}_v}(n)$ and $\text{Val}_v(\psi)=\sum_{n \in \mathbb{Z}^2}e^{in\cdot \psi}\widehat{\text{Val}_v}(n)$. At this point we introduce more definition. By an $\textit{advanced derivative tree}$ we mean a derivative tree with an extra label $n_v \in \mathbb{Z}^2$ on each of the nodes. Two advanced derivative trees are equal if and only if they are equal as derivative trees, and all the $n_v$ labels agree. If $\vartheta$ is an advanced derivative tree, we let $n(\vartheta)=\sum_{v \in V(\vartheta)}n_v$. Let $\widetilde{\Theta}_{k,n,\alpha}$  be the set of all advanced derivative trees $\vartheta$ with $k$ nodes such that the label on the top node is $\alpha$, $n(\vartheta)=n,$ and as derivative trees, $\vartheta \in \widetilde{H\partial}(\Theta_{k,\alpha}):= \bigcup_{E \in H \partial (\Theta_{k,\alpha})} \bigcup_{\vartheta '\in \Tilde{E}} \vartheta'$. \\
     \indent Now note that $\text{Val}[q_k(0)]=\sum_{n \in \mathbb{Z}^2}e^{in \cdot \psi}  \widehat{\text{Val}[q_k(0)]}(n),$ where 
\begin{align}
\label{fourier of q_k(0)}
\widehat{\text{Val}[q_k(0)]}(n) &=\sum_{\vartheta \in \widetilde{\Theta}_{k,n,\alpha}} \left[ \prod_{\substack{v \in V(\vartheta) \\ v \notin M \vartheta}} \left( \frac{\alpha_v}{s_v!}\lambda_{\alpha_v}^{-|p_v+1|\alpha_v} \prod_{j=1}^{s_v} \left(iv_{\alpha_{v_j}} \cdot S_0^{-p(v)} n_v\right) \hat{f}_{\alpha_v}(S_0^{-p(v)}n_v)  \right) \cdot \right. \notag \\
& \left. \prod_{\substack{v \in M \vartheta}} \left( \frac{\alpha_w}{s_w!}\lambda_{\alpha_v}^{-|p_v+1|\alpha_v+p_w} \left(iv_{-} \cdot S_0^{-p(v)} n\right) \prod_{j=1}^{s_v} \left(iv_{\alpha_{v_j}} \cdot S_0^{-p(v)} n_v\right) \hat{f}_{\alpha_v}(S_0^{-p(v)}n_v) \right) \right] .
\end{align}
Let $F=\text{max}_{\psi \in \mathbb{T}^2} |f_{\pm}(\psi)|$. Note that since $f$ is a trigonometric polynomial, at most $(2N+1)^2$ terms in the summand of (\ref{fourier of q_k(0)}) is non-zero. Since $|v_{\pm}|=1$, when the terms are non-zero it is bounded by $(2N)^k$. Thus we get 
\begin{equation}
\label{bound after fourier}
   |\text{Val}[q_k(0)]|=\sum_{n \in \mathbb{Z}^2}|\widehat{\text{Val}[q_k(0)]}(n)| \leq (2N)^k\cdot F^k \cdot (2N+1)^k \cdot 2^k \cdot 2^{2k}\cdot k \cdot \frac{\lambda_+^{-2k} [A_k(\lambda_+^{-2})]^k}{(1-\lambda_+^{-2})^{k(k+1)}}, 
\end{equation}

where the factor $2^k$ comes from the possibilities of having  + or $-$ label on each node, $2^{2k}$ is the max number of trees with $k$ nodes. The factor $k$ is to included since the derivative label could be on each of the node (except for the top node). The last factor is computed in (\ref{S_c bound}), which is obtained from summing the lambda factors over all the integer labels first. Thus by elementary theory of power series it is easy to see that $\sum_{k=0}^{\infty} \epsilon ^k \text{Val} [q_k(0)]$ has a positive radius of convergence.
\end{proof}

\begin{proposition} For every $n \geq 1,$ $\lim_{t \rightarrow 0} q_n(t)=\text{Val}[q_n(0)].$
\end{proposition}
\begin{proof}
     First we find a smarter way to write $q_n(t)$ so our previous tree cancellations becomes relevant. If $\vartheta$ is a derivative tree with derivative label on the node $d$, let $\leq$ be the ordering on $V(\vartheta)$ given by the Breadth First Search algorithm. \cite{alg} (Note that the ordering is unique for a fixed tree). Now we set 
    \[\text{Val}_t(\vartheta,\psi)=  \prod_{\substack{w \in V(\vartheta) \\ w < d}}\text{Val}(w,\psi) \cdot \left[\frac{\text{Val}(d,\psi+tv_-)-\text{Val}(d,\psi)}{t}\right] \cdot \prod_{\substack{w \in V(\vartheta) \\ w > d}}\text{Val}(w,\psi+tv_-).\]
Note that
    \begin{align}
         V_{\alpha}^{(m)} (\psi, t) &=\frac{h_{\alpha}^{(m)}(\psi+tv_-)-h_{\alpha}^{(m)}(\psi)}{t} =\sum_{\vartheta \in \Theta_{m,\alpha}}\frac{\text{Val}(\vartheta,\psi+tv_-)-\text{Val}(\vartheta,\psi)}{t} \notag \\
        &= \sum_{\vartheta \in \Theta_{m,\alpha}} \frac{1}{t} \left (\prod_{v \in V(\vartheta)} \text{Val}(v,\psi+tv_-)- \prod_{v \in V(\vartheta)} \text{Val}(v,\psi)\right) = \sum_{\theta \in \partial(\Theta_{m,\alpha})} \text{Val}_t (\theta, \psi).
    \end{align}
     If $p=\vartheta_1 \times \cdots \times \vartheta_n$ is a product tree, we set $\text{Val}_t(p, \psi):=\prod_{i=1}^{n}\text{Val}_t(\vartheta_i,\psi)$. The key observation is that, at $\psi \in \mathbb{T}^2,$
    \begin{equation}
    \label{q_n(t) formula}
        q_n(t)= \sum_{p \in \mathcal{P}_{n,+}} (-1)^{s(p)}\text{Val}_t(p,\psi).   
    \end{equation}
    (Compare (\ref{tree representation for q_0})).  Note that for $p\in \mathcal{P}_{n,+} $, $v\in V(p)$, since the F contribution $F_v$ on $v$ has the property that $t \rightarrow \text{Val}(v,\psi+tv_-)$ is analytic in $t$, we have
    \begin{equation}
    \label{estimate}
        \frac{\text{Val}(v,\psi+tv_-)-\text{Val}(v,\psi)}{t}=\partial_-\text{Val}(v,\psi)+\sum_{i=1}^{\infty}\frac{t^{i}}{(i+1)!}\partial_-^{i+1} \text{Val}(v,\psi).
    \end{equation}
     Clearly the second term on the right hand side is of order $t$ as $t\rightarrow 0$. If adapt the common notation $O(t)$ to represent any function of order $t$ as $t \rightarrow 0$ (i.e., there exist an $\epsilon>0$, $C \in [0,\infty)$ so that for every $t \in (-\epsilon, \epsilon)$, $|O(t)/t| \leq C $), then
    \begin{align}
    \label{Val_t}
        \text{Val}_t(p,\psi) &=\left[\partial_- F(\psi)+O(t)\right] \cdot \prod_{\substack{w \in V(\vartheta) \\ w < d}}\text{Val}(w,\psi) \cdot \prod_{\substack{w \in V(\vartheta) \\ w > d}}\text{Val}(w,\psi+tv_-) \notag \\
        & = \left[\partial_- F(\psi)+O(t)\right] \cdot \prod_{\substack{w \in V(\vartheta) \\ w \neq d}} \left [\text{Val}(w,\psi) + O(t) \right] \notag \\
        & =\text{Val}(p,\psi)+O(t),
    \end{align} 
    thus $q_n(t)=\sum_{p \in \mathcal{P}_{n,+}} (-1)^{s(p)} [\text{Val}(p,\psi)+O(t)]$. Now we perform the exact procedure of re-grouping, re-indexing, and cancellation (up to order $t$) as did in Lemma 4 to the series $q_n(t)$, and then send $t$ to 0. It is clear that the result we get is equal to $\text{Val}[q_n(0)]$, because whenever we are doing an change of index, the series on which this is performed always converges absolutely, so we do not run into issues. Moreover, the change of limit and summation is due to the Dominated Convergence Theorem.
\end{proof}

The final step to finish the proof of theorem \ref{main theorem} is that one have to show the series $q(\epsilon, t)=\sum_{i=1}^{\infty}\varepsilon ^i q_i(t)$ converges within some range $t \in (0, t^*)$ for which the radius of convergence is some uniform value $D.$ This has to be done in order to ensure the procedure of taking $t$ to 0 in $q(\epsilon, t)$ is valid, and gives $v_{\varepsilon}(\psi)=\sum_{k=1}^{\infty} \varepsilon^k \text{Val}[q^{(k)}(0)]$. This could be easily done by bounding the remainder $q_n(t)-\text{Val}(q_n(0))$ using the method as described in Proposition \ref{converge prop}.

\section{The Three Dimensional Case}
We conclude our discussion with some remarks about the three-dimensional case. Our method of formal series and product trees gives a good method for finding the tangent vector field of the stable and unstable manifold of the analytic perturbation of a special class of linear automorphisms on $\mathbb{T}^2$. On $\mathbb{T}^3$, however, some of the constructions we had is not valid. We first make some definations.
\begin{definition}
    A diffeomorphism $S: \mathbb{T}^3 \rightarrow \mathbb{T}^3$ is said to be $\textbf{partially hyperbolic}$ if there exists three 1-dimensional vector bundles $E^u,E^c,E^s$ that is invariant under $DS$ with the property that \\
    \indent (i) $T \mathbb{T}^3 = E^u \oplus E^c\oplus E^s$; \\
    \indent (ii) There are continuous functions $\sigma, \mu: \mathbb{T}^3 \rightarrow \mathbb{R}$ such that $0<\sigma <1<\mu$ and for every $p\in \mathbb{T}^3$ and unit vectors $v^* \in E^*(p)$ where $*=s,c,u$, one has $||Df(v^s)|| < \sigma(p) < ||Df(v^c)||<\mu(p)<||Df(v^u)||$.
\end{definition}
As an example \cite{GogolevMaimonKolmogorov2019}, consider the automorphism given by the following matrix with determinant 1:
\[
A = \begin{pmatrix}
2 & 1 & 0 \\
1 & 2 & 1 \\
0 & 1 & 1 \\
\end{pmatrix}
\]
whose eigenvalues are approximately 0.20, 1.55, and 3.25. $A$ could be viewed as a partially hyperbolic system with expanding center distribution. In fact, $A$ is also Anosov, and one could construct, as did in Proposition \ref{conjugation theorem}, the conjugation function  $H_{\varepsilon}$ between $A$ and $A_{\varepsilon}:=A+\varepsilon f.$ ($f$ is again a trigonometric polynomial.) If we denote $v^*$, $ *=s,c,u$ the three normalized eigenvectors, then it is easy to see that $\{H_{\varepsilon}(\psi+tv^s) \pmod {2\pi}: t\in \mathbb{R} \}$ is the stable manifold and $\{H_{\varepsilon}(\psi+tv^u+sv^c) \pmod {2\pi}: t,s \in \mathbb{R} \}$ is the unstable manifold of $A_{\varepsilon}$. Thus they are dense in $\mathbb{T}^3$ as it is the image of a dense set under a continous map. \\
\indent The same method described in this paper shows that one could "differentiate" the parameterization $t \mapsto H_{\varepsilon}(\psi+tv^s)$ and $t \mapsto H_{\varepsilon}(\psi+tv^c),$ but fails for the parameterization $t \mapsto H_{\varepsilon}(\psi+tv^u)$. It is known from \cite{RenGanZhang2017} that there exist foliations $W^s$ and $W^u$ tangent to $E^s$ and $E^u$ for $A_{\varepsilon}$, which suggest that the parametrization $t \mapsto H_{\varepsilon}(\psi+tv^u)$ can not be a parametrization of $W^u$. This suggests that one can not naively expect that the image, under the conjugation, of the invariant foliations for the unperturbed system gives the invariant foliations for the perturbed one. In fact, this happens precisely if $E^s \oplus E^u$ is integrable \cite{RenGanZhang2017, GanShi2019}.

    

\bibliographystyle{ieeetr}  
\bibliography{references}  
\newpage

\section{Appendix}

\begin{lemma}
Given sequences $\{a_n\}$ and $\{b_n\}$, where $a_0 = b_0 = 1$. We define a sequence $\{q_n\}$ such that $q_0 = 1$, and satisfying the formal power series product
$$\left(\sum_{k = 0}^{\infty} \epsilon^{k}b_k \right)\left(\sum_{k = 1}^{\infty} \epsilon^{k}q_k \right) = \left (\sum_{k = 1}^{\infty} \epsilon^{k}a_k \right).$$
Then, for $n \in \mathbb{N}$:
$$q_n = a_n - \sum_{d = 1}^{n-1} q_db_{n-d}$$
\end{lemma}

\begin{proof}
Equating terms with the same coefficient of $\epsilon$, 
$$\epsilon^k q_n + \epsilon^k \sum_{d = 1}^{n-1}q_d b_{n-d} = \epsilon^k a_n$$
Cancelling out $\epsilon^k$ and simplifying gives the desired result.
\end{proof}

\begin{lemma}
\label{explicit formula for power series}
With the same notation as in the previous lemma, the explicit formula for $q_n$ is 
$$q_n = a_n + \sum_{k = 1}^{n - 1}a_k \left(\sum_{s = 0}^{n - k}(-1)^s \sum_{\substack{m_1 + ... + m_s = n - k \\ m_1,...,m_s \geq 1}} (b_{m_1} b_{m_2} ... b_{m_s}) \right).$$
\end{lemma}
\begin{proof}
We will prove by induction. The base case is trivial; for both the recursive formula and the explicit formula, the large sums in both equations become $0$, and we are left with $q_1 = a_1$. For the inductive step, for any $n \in \mathbb{N}$ let us assume that:
$$q_d = a_d + \sum_{k = 1}^{d - 1}a_k \left(\sum_{s = 0}^{d - k}(-1)^s \sum_{\substack{m_1 + ... + m_s = d - k \\ m_1,...,m_s \geq 1}} (b_{m_1} b_{m_2} ... b_{m_s}) \right)$$
For all $d < n$. We now wish to prove that the explicit formula is true for $q_n$.
We do this by substituting our inductive hypothesis into our recursive formula:
\begin{align}
    q_n & = a_n - \sum_{d = 1}^{n-1} q_db_{n-d} \notag \\
    & = a_n - \sum_{d = 1}^{n-1} b_{n-d} \left(a_d + \sum_{k = 1}^{d - 1}a_k \left(\sum_{s = 0}^{d - k}(-1)^s \sum_{\substack{m_1 + ... + m_s = d - k \\ m_1,...,m_s \geq 1}}  (b_{m_1} b_{m_2} ... b_{m_s}) \right) \right) \notag
\end{align}
The above formula looks downright repugnant. In order to prove it, we require another sub-lemma, whose proof is quite trivial:
\begin{sub-lemma}
Let $P(n)$ be the multiset (a set that allows repetitions) of all ordered partitions of $n$, written as products of $b_i$ (like in the sum). For example: \\
$$P(3) = \{b_1 b_1 b_1, b_1 b_2, b_2 b_1, b_3\}$$
We will also define $P(0) = \{1\}$.\\
We define the multiset $P(n) + k$ as
$$P(n) + k = \{b_k a | a \in P(n)\}$$
For example, $P(3) + 2 = \{b_2 b_1 b_1 b_1, b_2 b_1 b_2, b_2 b_2 b_1, b_2b_3\}$.
Then:
$$P(n) = (P(n - 1) + 1)) \cup ((P(n - 2) + 2) \cup ... \cup (P(1) + (n-1)) \cup (P(0) + n)$$
From this fact, we can imply that
$$\sum_{k \in P(n)} k = \sum_{k \in (P(n - 1) + 1)}k \quad + \sum_{k \in (P(n - 2) + 2)} k \quad + ...+ \sum_{k \in (P(1) + (n -1))} k.$$
\end{sub-lemma}
We are ready to finish our proof. First, we note that 
$$\left(\sum_{s = 0}^{d - k}(-1)^s \sum_{\substack{m_1 + ... + m_s = d - k \\ m_1,...,m_s \geq 1}} (b_{m_1} b_{m_2} ... b_{m_s})\right)=\sum_{w \in P(d - k)} \alpha w,$$
where $\alpha = 1$ if the number of $b_i$ in $w$ is even, and $-1$ otherwise. Substituting this expression into our large equation for $q_n$, we finally have:
$$q_n = a_n - \sum_{d = 1}^{n-1} b_{n-d} \left(a_d + \sum_{k = 1}^{d - 1}a_k \left(\sum_{w \in P(d - k)} \alpha w \right)\right)$$
Distributing the term $\sum_{d = 1}^{n-1} b_{n - d}$, we get:
\begin{align}
    q_n &= a_n -  \left(\sum_{d = 1}^{n-1} b_{n-d} a_d + \left(\sum_{d = 1}^{n-1}  \left(\sum_{k = 1}^{d - 1}b_{n-d} a_k \left(\sum_{w \in P(d - k)} \alpha w \right) \right) \right) \right) \notag \\
    & = a_n -  \left(\sum_{d = 1}^{n-1} b_{n-d} a_d + \left(\sum_{k = 1}^{n-1} \left(\sum_{d =  k + 1}^{n - 1} a_k b_{n-d} \left(\sum_{w \in P(d - k)} \alpha w \right) \right) \right) \right) \notag
\end{align}
In the term $(\sum_{d =  k}^{n - 1} b_{n-d} (\sum_{w \in P(d - k)} \alpha w))$, we are multiplying each $w$ by a $b_{n - d}$. This turns the sum into the following:
$$\sum_{d =  k + 1}^{n - 1} \left(\sum_{w \in (P(d - k) + (n - d))} \alpha w \right)=\sum_{w \in P(n - k)} \alpha w - b_{n-k},$$
Since we have incremented the number of $b_i$ in $w$ by one, $\alpha = 1$ when the number of $b_i$ is odd, and $-1$ otherwise. Since our index does not start at $d = k$, and thus we do not get the set $P(0) + (n - k) = \{b_{n-k}\}$.Plugging it back into our equation, we get:
\begin{align}
    q_n &= a_n -  \left(\sum_{d = 1}^{n-1} b_{n-d} a_d + \left(\sum_{k = 1}^{n-1} a_k \left(\sum_{w \in P(n - k)} \alpha w - b_{n-k} \right) \right) \right) \notag \\
    &=a_n -  \left(\sum_{d = 1}^{n-1} b_{n-d} a_d - \sum_{k = 1}^{n-1} a_kb_{n-k} + \sum_{k = 1}^{n-1} a_k \left(\sum_{w \in P(n - k)} \alpha w \right)\right) \notag \\
    &=a_n -  \left(\sum_{k = 1}^{n-1} a_k \left(\sum_{w \in P(n - k)} \alpha w \right) \right) \notag
\end{align}

We multiply the negative sign in so that, again, $\alpha = 1$ when the number of $b_i$ in $w$ are even, and $-1$ otherwise. Re-expanding, we finally get:
$$q_n = a_n + \sum_{k = 1}^{n - 1}a_k \left(\sum_{s = 0}^{n - k}(-1)^s \sum_{\substack{m_1 + ... + m_s = n - k \\ m_1,...,m_s \geq 1}} (b_{m_1} b_{m_2} ... b_{m_s}) \right)$$ 
\end{proof}

\end{document}